\documentclass[12pt]{amsart}
\usepackage{amsmath,amssymb,amsfonts,amsthm}
\usepackage[dvips]{graphicx,epsfig,color}
\usepackage{enumerate}
\usepackage[mathscr]{eucal}
\usepackage{float}
\numberwithin{equation}{section}
\newtheorem{theorem}{Theorem}[section]
\newtheorem{definition}{Definition}[section]
\newtheorem{lemma}[theorem]{Lemma}
\newtheorem{proposition}[theorem]{Proposition}
\newtheorem{corollary}[theorem]{Corollary}
\newtheorem{note}[theorem]{Note}
\numberwithin{equation}{section}

 \makeatletter
      \def\@setcopyright{}
      \def\serieslogo@{}
      \makeatother

\begin{document}
\bibliographystyle{plain}

\title{Chromatic Completion Number}
\author{$^*$E.G. Mphako-Banda and $^{\dag}$J. Kok }
\address{ $^*$School of Mathematical Sciences, University of Witswatersrand, Johannesburg, South Africa.\\ 
 $^{\dag}$Centre for Studies in Discrete Mathematics,Vidya Academy of Science \& Technology,Thrissur, India}
\email{$^*$eunice.mphako-banda@wits.ac.za and $^{\dag}$ kokkiek2@tshwane.gov.za}
\date{}

\begin{abstract}
We use a well known concept of proper vertex colouring of a graph to introduce the construction of a chromatic completion graph and its related parameter, the chromatic completion number of a graph. We then give the  chromatic completion number of certain classes of cycle derivative graphs and helm graphs. Finally, we discuss further problems for research related to this concept.  
\end{abstract}

\maketitle

\section{Introduction}
For general notation and concepts in graphs see \cite{bondy,harary,banda}.  Unless stated otherwise, all graphs will be finite and simple, connected graphs with at least one edge. The set of vertices and the set of edges of a graph $G$ are denoted by, $V(G)$ and  $E(G)$ respectively. The number of vertices is called the order of $G$ say, $n$ and the number of edges of $G$ is denoted by, $\varepsilon(G).$ If $G$ has order $n \geq 1$ and has no edges ($\varepsilon(G)=0$) then $G$ is called a null graph. The degree of a vertex $v \in V(G)$ is denoted $d_G(v)$ or when the context is clear, simply as $d(v)$. The minimum and maximum degree $\delta(G)$ and $\Delta(G)$ respectively, have the conventional meaning. When the context is clear we shall abbreviate to $\delta$ and $\Delta,$ respectively.

For a set of distinct colours $\mathcal{C}= \{c_1,c_2,c_3,\dots,c_\ell\}$ a vertex colouring of a graph $G$ is an assignment $\varphi:V(G) \mapsto \mathcal{C}.$ A vertex colouring is said to be a \textit{proper vertex colouring} of a graph $G$ if no two distinct adjacent vertices have the same colour. The cardinality of a minimum set of distinct colours in a proper vertex colouring of $G$ is called the \textit{chromatic number} of $G$ and is denoted $\chi(G).$ We call such a colouring a $\chi$-colouring or a \textit{chromatic colouring} of $G.$ A chromatic colouring of $G$ is denoted by $\varphi_\chi(G)$. Generally a graph $G$ of order $n$ is $k$-colourable for $\chi(G) \leq k.$ Unless mentioned otherwise, a set of colours will mean a set of distinct colours.

Generally the set, $c(V(G)) \subseteq \mathcal{C}.$ A set $\{c_i \in \mathcal{C}: c(v)=c_i$ for at least one $v\in V(G)\}$ is called a colour class of the colouring of $G.$ If $\mathcal{C}$ is the chromatic set it can be agreed that $c(G)$ means set $c(V(G))$ hence, $c(G) \Rightarrow \mathcal{C}$ and $|c(G)| = |\mathcal{C}|.$ For the set of vertices $X\subseteq V(G),$ the subgraph induced by $X$ is denoted by, $\langle X\rangle.$ The colouring of $\langle X\rangle$ permitted by $\varphi:V(G) \mapsto \mathcal{C}$ is denoted by, $c(\langle X\rangle).$ The number of times a colour $c_i$ is allocated to vertices of a graph $G$ is denoted by $\theta_G(c_i)$ or if the context is clear simply, $\theta(c_i).$

Index labeling the elements of a graph such as the vertices say,\\ $v_1,v_2,v_3,\dots,v_n$ or written as, $v_i,$ where $i = 1,2,3,\dots,n,$ is called minimum parameter indexing. Similarly, a \textit{minimum parameter colouring} of a graph $G$ is a proper colouring of $G$ which consists of the colours $c_i;\ 1\le i\le \ell.$

In this paper, Section~\ref{s2} introduces a new parameter called, the \emph{chromatic completion number} of a graph $G.$ Subsection~\ref{sub2.1} presents results on chromatic completion number for a few known classes of cycle derivative graphs. Subsection~\ref{sub2.2} presents results on chromatic completion number on helm graphs. Finally, in section~\ref{s3}, a few suggestions on future research on this problem are discussed.
\section{Chromatic completion number of cycle derivative graphs}
\label{s2}
In an improper colouring an edge $uv$ for which, $c(u)=c(v)$ is called a \emph{bad edge}. See \cite{banda} for an introduction to $k$-defect colouring and corresponding polynomials. For a colour set $\mathcal{C},$ $|\mathcal{C}| \geq \chi (G)$ a graph $G$ can always be coloured properly hence, such that no bad edge results. Also, for a set of colours $\mathcal{C},$ $|\mathcal{C}| = \chi (G) \geq 2$ a graph $G$ of order $n$ with corresponding chromatic polynomial $\mathcal{P}_G(\lambda),$ can always be coloured properly in $\mathcal{P}_G(\lambda)$ distinct ways.  The notion of the \emph{chromatic completion number} of a graph $G$ denoted by, $\zeta(G)$ is the maximum number of edges over all chromatic colourings that can be added to $G$ without adding a bad edge. The resultant graph $G_\zeta$ is called a \emph{chromatic completion graph} of $G.$ The additional edges are called \emph{chromatic completion edges}. It is trivially true that $G\subseteq G_\zeta.$ Clearly for a complete graph $K_n,$ $\zeta(K_n)=0.$ In fact for any complete $\ell$-partite graph $H=K_{n_1,n_2,n_3,\dots,n_\ell},$ $\zeta(H)=0.$ Hereafter, all graphs will not be $\ell$-partite complete. For graphs $G$ and $H$  of order $n$ with $\varepsilon(G)\geq \varepsilon(H)$ no relation between $\zeta(G)$ and $\zeta(H)$ could be found. The first result is straight forward.
\begin{theorem}
\label{thm2.1}
A graph $G$ of order $n$ is not complete, if and only if $G_\zeta$ is not complete.
\end{theorem}
\begin{proof}
Let $G$ be of order $n,$ then $G_\zeta$ is of order $n.$ If $G_\zeta \ncong K_n$ then  $G\ncong K_n,$ since $G\subseteq G_\zeta.$

Conversely, if $G$ is not complete then $\chi(G)< n$ hence, for any chromatic colouring of $G,$ at least one pair of distinct vertices say $u$ and $v$ exists such that $c(u)=c(v).$ Therefore, edge $uv\notin E(G_\zeta)$  implying $G_\zeta$ is not complete.
\end{proof}

Theorem~\ref{thm2.1} can be stated differently i.e. $G$ is complete if and only if $G_\zeta$ is complete.
The next lemma does not necessarily correspond to a chromatic completion graph. It represents a \emph{pseudo completion graph} corresponding to a chromatic colouring, $\varphi:V(G)\mapsto \mathcal{C}.$
\begin{lemma}
\label{lem2.2}
For a chromatic colouring $\varphi:V(G)\mapsto \mathcal{C}$ a pseudo completion graph, $H(\varphi)= K_{n_1,n_2,n_3,\dots,n_\chi}$ exists such that, $$\varepsilon(H(\varphi))-\varepsilon(G) =\sum\limits_{i=1}^{\chi-1}\theta_G(c_i)\theta_G(c_j)_{(j=i+1,i+2,i+3,\dots,\chi)}-\varepsilon(G) \leq \zeta(G).$$
\end{lemma}
\begin{proof}
For any chromatic colouring $\varphi:V(G)\mapsto \mathcal{C},$ the graph, $H(\varphi) = K_{\theta_G(c_1),\theta_G(c_2),\dots,\theta_G(c_\chi)}$ is a corresponding pseudo completion graph. Therefore the result as stated.
\end{proof}
Now we are ready for a main result in the form of a corollary which is a direct consequence of Lemma~\ref{lem2.2}
\begin{corollary}
\label{col2.3}
Let  $G$ be a graph. Then 
\begin{eqnarray*}
\zeta(G) & =& max(\varepsilon(H(\varphi)) -\varepsilon(G) \text{ over all} \ \varphi:V(G)\mapsto \mathcal{C}.
\end{eqnarray*}
\end{corollary}
\begin{theorem}
\label{thm2.4}
Let  $G$ be a graph.  Then $\zeta(G)\leq \varepsilon(\overline{G})$, and equality holds if and only if $G$ is complete.
\end{theorem}
\begin{proof}
Since a chromatic completion edge $e\notin E(G)$ it follows $e\in E(\overline{G})$ hence, $\zeta(G)\leq \varepsilon(\overline{G}).$
\end{proof}
An immediate consequence of Theorem 2.4 read with the definition of chromatic completion is that equality holds for a graph $G$ if and only if, for all pairs of distinct vertices, $u$, $v$ for which the edge, $uv \notin E(G)$ we have, $c(u)\neq c(v)$.

For a positive integer $n \geq 2$ and $2\leq \ell \leq n$ let integers, \\$1\leq a_1,a_2,a_3,\dots,a_{\ell-r}, a'_1,a'_2,a'_3,\dots,a'_r \leq n-1$ be such that\\ $n=\sum\limits_{i=1}^{\ell-r}a_i + \sum\limits_{j=1}^{r}a'_j.$ Then $(a_1,a_2,a_3,\dots,a_{\ell-r}, a'_1,a'_2,a'_3,\dots,a'_r)$ is called a $\ell$-partition of $n$ and $\sum\limits_{i=1}^{\ell-r-1}\prod\limits_{k=i+1}^{\ell-r}a_ia_k + \sum\limits_{i=1}^{\ell-r}\prod\limits_{j=1}^{r}a_ia'_j + \sum\limits_{j=1}^{r-1}\prod\limits_{k=j+1}^{r}a'_ja'_k$ is called the \emph{sum of permutated term products} of the $\ell$-partition of $n.$

To illustrate the concepts consider $n=2.$ Since, $(1,1)$ is the only $2$-partition of 2, it follows that $1\times 1=1$ is the only sum of permutated term product, (a single product for $n=2$). For $n=5$ and by the commutative law there are two distinct possible $3$-partitions namely, $(1,1,3)$ or $(1,2,2).$ Hence, the two distinct sum of permutated term products are equal to 7 and 8.  For $n=8$ and by the commutative law there are four distinct possible $3$-partitions namely, $(1,2,5),$ $(1,3,4),$ $(2,3,3)$ or $(2,2,4),$ with corresponding sum of permutated term products equal to 17, 19, 21 and 20, respectively.
\begin{definition}
\label{def2.1}
{\rm For two positive integers $2\leq \ell \leq n$ the division, $\frac{n}{\ell} = \lfloor\frac{n}{\ell}\rfloor + r,$ with $r$ some positive integer and $\ell> r\geq 0$. Hence, $n= \underbrace{\lfloor\frac{n}{\ell}\rfloor+\lfloor\frac{n}{\ell}\rfloor+\cdots+\lfloor\frac{n}{\ell}\rfloor}_{(\ell-r)-terms} +\underbrace{\lceil\frac{n}{\ell}\rceil +\lceil\frac{n}{\ell}\rceil+\cdots+\lceil\frac{n}{\ell}\rceil}_{(r\geq 0)-terms}.$ This specific $\ell$-partition, $(\underbrace{\lfloor\frac{n}{\ell}\rfloor,\lfloor\frac{n}{\ell}\rfloor,\dots,\lfloor\frac{n}{\ell}\rfloor}_{(\ell-r)-terms},\underbrace{\lceil\frac{n}{\ell}\rceil,\lceil\frac{n}{\ell}\rceil,\dots,\lceil\frac{n}{\ell}\rceil}_{(r\geq 0)-terms})$ is called a \emph{completion $\ell$-partition} of $n.$}
\end{definition} 
The next theorem is a number theoretical result which finds application in the study of chromatic completion of graphs. To ease the formulation of the next result let, $t_i=\lfloor\frac{n}{\ell}\rfloor,$ $i=1,2,3,\dots,(\ell-r)$ and $t'_j=\lceil\frac{n}{\ell}\rceil,$ $j=1,2,3,\dots,r.$ Call, $ \mathcal{L}=\sum\limits_{i=1}^{\ell-r-1}\prod\limits_{k=i+1}^{\ell-r}t_it_k + \sum\limits_{i=1}^{\ell-r}\prod\limits_{j=1}^{r}t_it'_j + \sum\limits_{j=1}^{r-1}\prod\limits_{k=j+1}^{r}t'_jt'_k,$ the \emph{$\ell$-completion sum-product} of $n.$
\begin{theorem}${(Lucky's~Theorem)}$\footnote{Dedicated to late Lucky Mahlalela who was a disabled, freelance traffic pointsman in the City of Tshwane. Sadly he was brutally murdered.}
For a positive integer $n \geq 2$ and $2\leq p \leq n$ let integers, $1\leq a_1,a_2,a_3,\dots,a_{p-r}, a'_1,a'_2,a'_3,\dots,a'_r \leq n-1$ be such that $n=\sum\limits_{i=1}^{p-r}a_i + \sum\limits_{j=1}^{r}a'_j$ then, the $\ell$-completion sum-product $\mathcal{L} = max\{\sum\limits_{i=1}^{p-r-1}\prod\limits_{k=i+1}^{p-r}a_ia_k + \sum\limits_{i=1}^{p-r}\prod\limits_{j=1}^{r}a_ia'_j + \sum\limits_{j=1}^{r-1}\prod\limits_{k=j+1}^{r}a'_ja'_k\}$ over all possible, $n=\sum\limits_{i=1}^{p-r}a_i + \sum\limits_{j=1}^{r}a'_j.$
\label{thm2.5}
\end{theorem}
\begin{proof}
Let $n,p \in \Bbb N$, $2 \leq p \leq n.$ The commutative law is valid for addition and multiplication hence, we assume that, $1 \leq a_1\leq a_2 \leq a_3 \leq\cdots \leq a_{p-r}\leq a'_1\leq a'_2\leq a'_3\leq \cdots \leq a'_r \leq n-1$ and that $n=\sum\limits_{i=1}^{p-r}a_i + \sum\limits_{j=1}^{r}a'_j.$

For $p=2,$ consider $a_1=x,$ $a'_1 =n-x.$ So $a_1\times a'_1=x(n-x)$ for which a maximum of $\frac{n}{2}\times \frac{n}{2}$ is obtain at $x= \frac{n}{2}.$ We restrict values to integer products thus an integer maximum is attained for the ordered pairs, $(\lfloor \frac{n}{2}\rfloor, \lfloor \frac{n}{2}\rfloor)$ or $(\lfloor \frac{n}{2}\rfloor, \lceil \frac{n}{2}\rceil)$ or $(\lceil \frac{n}{2}\rceil, \lceil \frac{n}{2}\rceil).$ Hence, the result, \emph{maximum sum of permutated term products} holds for the completion 2-partition of $n$ if $p=2.$ Assume it holds for $p = q \in \Bbb N.$ Hence, the assumption states that for, $1 \leq a_1,a_2,a_3,\dots,a_{q-r}, a'_1,a'_2,a'_3,\dots,a'_r \leq n-1$ be such that:

$n=\sum\limits_{i=1}^{q-r}a_i + \sum\limits_{j=1}^{r}a'_j$ then, the $q$-completion sum-product \\ $\mathcal{L} = max\{\sum\limits_{i=1}^{q-r-1}\prod\limits_{k=i+1}^{q-r}a_ia_k + \sum\limits_{i=1}^{q-r}\prod\limits_{j=1}^{r}a_ia'_j + \sum\limits_{j=1}^{r-1}\prod\limits_{k=j+1}^{r}a'_ja'_k\}$ over all possible, $n=\sum\limits_{i=1}^{q-r}a_i + \sum\limits_{j=1}^{r}a'_j.$ Put differently, the aforesaid means that the \emph{sum of permutated term products} is a maximum over that particular $q$-partition. Hence, $(a_{i_{(1 \leq i \leq (q-r))}}, a'_{j_{(1 \leq j \leq r)}})$ corresponds to the completion $q$-partition of $n$ such that, 
\begin{eqnarray*}
n&=& \underbrace{\lfloor\frac{n}{q}\rfloor+\lfloor\frac{n}{q}\rfloor+\cdots+\lfloor\frac{n}{q}\rfloor}_{(q-r)-terms}
+ \underbrace{\lceil\frac{n}{q}\rceil +\lceil\frac{n}{q}\rceil+\cdots+\lceil\frac{n}{q}\rceil}_{(r\geq 0)-terms}.
\end{eqnarray*}
Now consider $p=q+1.$

Case 1: If $r>0$ for $\frac{n}{q},$ determine a $(q+1)^{th}$ sum-term by reducing a sufficient number of the $\lceil \frac{n}{q}\rceil$ sum-terms by $1$ each to obtain terms of the form $\lfloor \frac{n}{q+1}\rfloor$ or $\lceil \frac{n}{q+1}\rceil.$ The aforesaid is always possible. Each pair of terms in the $(q+1)$-partition corresponds to a $2$-completion sum-product of $\lfloor \frac{n}{q+1}\rfloor,$ $\lfloor \frac{n}{q+1}\rfloor$ or $\lfloor \frac{n}{q+1}\rfloor,$ $\lceil \frac{n}{q+1}\rceil,$ or $\lceil \frac{n}{q+1}\rceil,$ $\lceil \frac{n}{q+1}\rceil,$ so it follows that the maximum sum of permutated term products has been obtained between all pairs (follows from the case $p = 2$). It follows that the sum of the maximums yields a maximum over the sum of pairwise products hence, a maximum sum of permutated term products has been obtained. Furthermore, the $(q+1)$-partition obtained corresponds to the terms required for a $(q+1)$-completion sum-product of $n.$ Therefore, the result holds for $p = q+1$ thus it holds for any $2\leq p \leq n$ for which $\frac{n}{q}$ has $r>0.$
 
Case 2: Through similar reasoning the results holds for $r = 0.$

Through immediate induction it follows that the result holds for all $n \in \Bbb N,$ $n \geq 2.$ That concludes the proof.
\end{proof}

Theorem~\ref{thm2.5} leads to a lemma in which each term in a sum-term partition corresponds to a distinct colour class. Hence, if the colours are $c_i,$ $1\leq i \leq \ell$ then, $\theta(c_i) = \lfloor \frac{n}{\ell}\rfloor$ or $\lceil \frac{n}{\ell}\rceil.$
\begin{lemma}
\label{lem2.6}
If a subset of $m$ vertices say, $X \subseteq V(G)$ can be chromatically coloured by $t$ distinct colours and if the graph structure permits such, then allocate colours as follows:
\begin{enumerate}[(a)]
\item  For $t$ vertex subsets each of cardinality $s= \lfloor \frac{m}{t}\rfloor$ allocate a distinct colour followed by:
\item Colour one additional vertex (from the $r\geq 0$ which are uncoloured), each in a distinct colour.
\end{enumerate}
This chromatic colouring permits the maximum number of chromatic completion edges between the vertices in $X$ amongst all possible chromatic colourings of $X$.
\end{lemma}
Lemma~\ref{lem2.6} can be applied to a set of vertices which induce a connected graph by assigning a proper colouring. Lemma~\ref{lem2.6} also has an interesting implication. This is stated as a corollary.
\begin{corollary}
\label{col2.7}
 Let  $G$ be a graph. Then
\begin{enumerate}[(i)]
\item  a chromatic completion graph $G_\zeta$ is not unique.
\item  a set of chromatic completion edges of maximum cardinality is not unique.
\end{enumerate}
\end{corollary}
Another interesting implication of Lemma~\ref{lem2.6} is that for any $n\in \Bbb N$ the complete $\ell$-partite graph of order $n$ given by $$K_{(\underbrace{\lfloor\frac{n}{\ell}\rfloor,\lfloor\frac{n}{\ell}\rfloor,\cdots,\lfloor\frac{n}{\ell}\rfloor}_{(\ell-r)-terms}, \underbrace{\lceil\frac{n}{\ell}\rceil,\lceil\frac{n}{\ell}\rceil,\cdots+\lceil\frac{n}{\ell}\rceil}_{(r\geq 0)-terms})}, \ \frac{n}{\ell} = \lfloor\frac{n}{\ell}\rfloor + r$$ with $r$ some positive integer and $r\geq 0,$ has maximum number of edges amongst all complete $\ell$-partite graph of order $n.$ Furthermore, it is a direct consequence from the proof of Theorem~\ref{thm2.5} that for those graphs which permit the colour allocation prescribed by Lemma~\ref{lem2.6}, the maximum number of chromatic completion edges between the vertices in $X$ amongst all possible chromatic colourings of $X$ are unique hence, well-defined.

It is important to note that not all graphs permit the colour allocation prescribed by Lemma~\ref{lem2.6}. For such graphs an \emph{optimal near-completion $\ell$-partition} is always possible. The optimal near-completion $\ell$-partition follows from the fact that for $n+1,$ even, we have that $1\times n<2\times (n-1)< 3\times (n-2)<\cdots < (\frac{n+1}{2})^2.$ Similarly for $n+1,$ odd, we have that, $1\times n<2\times (n-1)< 3\times (n-2)<\cdots < \lfloor \frac{n+1}{2}\rfloor \times \lceil \frac{n+1}{2}\rceil.$ This then yields the unique chromatic completion number. See discussion following Proposition~\ref{prop2.8}.
\subsection{Chromatic completion number of certain graphs}
\label{sub2.1}
The result for acyclic graphs and even cyclic graphs (graphs containing only even cycles), $G$ of order $n$ is straight forward i.e. $\zeta(G) = \theta(c_1)\theta(c_2)-\varepsilon(G).$ Example, for an even cycle graph $C_n$ it follows that, $\zeta(C_n)=\frac{n}{2}\times \frac{n}{2}-n=\frac{n(n-4)}{4}.$ This section will henceforth, unless stated otherwise, consider graphs which contains at least one odd cycle, thus graphs for which $\chi(G)\geq 3.$ 

Let the vertices of a cycle graph $C_n$ be labeled $v_i$, $1\leq i \leq n.$ 
A sunlet graph $Sl_n,$ $n\geq 3$ is obtained from a cycle graph $C_n$ by attaching a pendant vertex $u_i$ to each cycle vertex $v_i,$ $1\leq i \leq n.$ A graph $W_{1+n} = C_n+K_1,$ $n\geq 3$ is called a wheel graph. The edges and vertices of $C_n$ are respectively, called rim edges and rim vertices. The vertex corresponding to $K_1$ is called the central vertex say, $v_0$  and the edges incident with the central vertex are called spokes.
 
Since a complete graph $K_n$ is obtain from a cycle graph $C_n$ by adding all possible chords, a complete graph is a cycle derivative graph as well. Recall that a sun graph $S_n,$ $n\geq 2$ is obtained from the complete graph $K_n$ by adding vertices $u_i$ and the edges $u_iv_i,$ $u_iv_{i+1},$ $1\leq i\leq n$ and where modular arithmetic at edge $v_nv_1$ has known meaning. Note that $S_2\cong K_3 \cong C_3$ and is therefore treated as $C_3.$
\begin{proposition}
\label{prop2.8}
\begin{enumerate}[(i)]
\item  Let $C_n$ be  an odd cycle graph  and  $n\geq 3.$ Then
\begin{eqnarray*} 
\zeta(C_n) &=&
\begin{cases}
n(\frac{n}{3}-1), &\text {if $n=0~(mod~3),$}\\
(n-2)\frac{n-5}{3} + \lceil \frac{n-2}{2}\rceil +1, &\text {if $n=2~(mod~3),$}\\
 (n-1)\frac{n-4}{3} + \lceil \frac{2}{3}(n-5)\rceil +1, & \text {if $n=1~(mod~3).$}
\end{cases}
\end{eqnarray*} 
\item  Let $Sl_n$ be a sunlet graph and  $n\geq 3.$ Then $\zeta(Sl_n)= 3\zeta(C_n)+n.$
\item  Let $W_{1,n}$  be a wheel graph  and $n\geq 3.$ Then
\begin{eqnarray*}  
\zeta(W_{1,n}) &=&
\begin{cases}
\frac{n^2}{4}, &\text {if $n$ is even,}\\
\zeta(C_n), & \text {if $n$ is odd.}
\end{cases}
\end{eqnarray*}
\item Let $S_n$  be a sun graph and $n\geq 3.$  Then $\zeta(S_n)= \frac{n(3n-4)}{2}.$
\end{enumerate} 
\end{proposition}
\begin{proof}
\begin{enumerate}[(i)]
\item Let $\Bbb{N}^{odd} = \{n:set~of~odd~positive~integers, n\geq 3\}.$ Let $\Bbb{N}_1 =\{n_i \in \Bbb{N}^{odd}: n_i=0~(mod~3)\},$ $\Bbb{N}_2 =\{n_j \in \Bbb{N}^{odd}: n_j=1~(mod~3)\},$ $\Bbb{N}_3 =\{n_k \in \Bbb{N}^{odd}: n_k=2~(mod~3)\}.$ Clearly, $\Bbb{N}^{odd} = \Bbb{N}_1 \cup \Bbb{N}_2 \cup \Bbb{N}_3.$

Part 1: Let $n =3t,$ $t=1,3,5,7,\dots.$ Hence, $n=0~(mod~3)$ and $\chi(C_n)=3.$ For the colour set $\mathcal{C} = \{c_1,c_2,c_3\}$ and without loss of generality and by symmetry consideration, the extremal number of vertices coloured $c_3$ are either $\theta_{C_n}(c_3) =1$ or $\theta_{C_n}(c_3) = \frac{n}{3}.$

Case 1; ($\theta_{C_n}(c_3)=1$): without loss of generality, let $c(v_n)=c_3.$ Note that, $C_n-v_n \cong P_{n-1}$ and $n-1$ is even. From Theorem~\ref{thm2.1} it follows that, $\varepsilon(H(\varphi)) -\varepsilon(C_n) =(n-1) +\frac{(n-1)^2}{4} -(n-2)-2=\frac{n^2-2n-3}{4}.$

Case 2; ($\theta_{C_n}(c_1)=\theta_{C_n}(c_2)=\theta_{C_n}(c_3)=\frac{n}{3}$): now $\varepsilon(H(\varphi)) -\varepsilon(C_n) = n(\frac{n}{3}-1).$

Since, for $n\geq 3$ it follows that, $n^2-6n+9\geq 0 \Rightarrow 4n^2-12n\geq 3n^2-6n-9 \Rightarrow \frac{n^2-3n}{3} \geq \frac{n^2-2n-3}{4}$ we have, $\zeta(C_n) \geq n(\frac{n}{3}-1).$ Through similar reasoning and immediate induction for $1 \leq \theta_{C_n}(c_3) < \frac{n}{3}$ it is concluded that, $\zeta(C_n) = n(\frac{n}{3}-1) =n(t-1).$

Part 2: Consider $C_n,$ $n =3t,$ $t=1,3,5,7,\dots$ as in (i)Part 1 with the extremal repetitive colouring, $c(v_1)=c_1,$ $c(v_2)=c_2,$ $c(v_3)=c_3,\cdots$, $c(v_n)=c_3.$ Now add vertex $v_{n+1},$ $v_{n+2}$ to obtain $C_{n+2}$ and note that the edge $v_nv_1$ is now a chord which represents a count of $+1.$  The additional vertices can only be coloured by the ordered pairs, $(c(v_{n+1}),c(v_{n+2})) = (c_1,c_2)$ or $(c_1,c_3)$ or $(c_2,c_3).$ The number of chromatic completion edges that can be added with an end vertex $v_{n+1}$ or $v_{n+2}$ is exactly $\lceil \frac{n}{2}\rceil.$ Hence, from (i)Part 1, $\zeta(C_{n+2}) = n(\frac{n}{3}-1) + \lceil \frac{n}{2}\rceil +1.$ Finally, standardising to the conventional notation gives the result for $n=2~(mod~3)$ i.e. $\zeta(C_n) = (n-2)(\frac{n-2}{3}-1) + \lceil \frac{n-2}{2}\rceil +1.$

Part 3:  Let $n =3s +1,$ $s=2,4,6,8,\dots.$ Hence, $n=1~(mod~3)$ and $\chi(C_n)=3.$ Similar to (i)Part 1 colour vertices $v_i,$ $1\leq i \leq n-1$ with the extremal repetitive colouring, $c(v_1)=c_1,$ $c(v_2)=c_2,$ $c(v_3)=c_3,\cdots ,$ $c(v_{n-1})=c_3.$ For the cycle graph $C_{n-1}$ it follows from (i)Part 1 that the chromatic completion number is $\zeta(C_{n-1})= (n-1)(\frac{n-1}{3}-1) =(n-1)\frac{n-4}{3}.$ In $C_n$ the edge $v_{n-1}v_1$ is a chord and corresponds to a count of $+1.$ The vertex $v_n$ can only be coloured $c_2.$ The number of chromatic completion edges from vertex $v_n$ is exactly $\lceil \frac{2}{3}(n-5)\rceil.$ Therefore, $\zeta(C_n) = (n-1)\frac{n-4}{3} + \lceil \frac{2}{3}(n-5)\rceil +1.$

\item  Colour the cycle subgraph as in (i). Colour the pendant vertices through say, a clockwise rotation of the cycle colouring of one vertex index that is, $c(v_i)\mapsto c(v_{i+1})$ and modular arithmetic for $v_n,$ $v_1$ has known meaning. Clearly the number of chromatic completion edges permitted amongst the pendant vertices \emph{per se} will be the chromatic completion edges of a cycle graph$C_n$ as well as, $\zeta(C_n)$ chromatic completion edges found for $C_n.$ Therefore, the partial count of chromatic completion edges permitted amongst the pendant vertices is, $\zeta(C_n)+n.$ The cycle graph itself permits $\zeta(C_n)$ chromatic completion edges. Finally, the number of chromatic completion edges permitted between the pendant vertices and the cycle vertices amounts to $\zeta(C_n)$ as well. Hence, $\zeta(Sl_n)=3\zeta(C_n)+n.$

\item Part 1: Because the central vertex is adjacent to all other vertices the chromatic completion edges can only come from the even rim cycle $C_n$. The result follows from Theorem~\ref{thm2.1}.

 Part 2: As in (i)Part 1, it follows that only the odd rim cycle can contribute to chromatic completion edges. Hence, the result follows from (i).
\item  For a complete graph $K_n,$ $n\geq 3$ each $v_i$ can uniquely be coloured $c_i,$ $1\leq i\leq n.$ From Lemma~\ref{lem2.6} it follows that each vertex $u_i$ can be uniquely coloured some $c_j,$ $c_j\neq c(v_i),$ $c_j\neq c(v_{i+1}),$ $1\leq i \leq n$ and where modular arithmetic at edge $v_nv_1$ has known meaning.  Because the set $\{u_i:1\leq i \leq n\}$ is an independent set and each vertex is uniquely coloured amongst the $u_i's$ he chromatic completion permits a complete graphs. This gives the number of chromatic completion edges to be $\frac{1}{2}n(n-1).$ Furthermore, each $u_i$ may be linked to a further $n-3$ vertices of $K_n.$ Hence, the total number of chromatic completion edges is, $\zeta(S_n) =\frac{1}{2}n(n-1) + n(n-3) = \frac{n(3n-4)}{2},$ $n\geq 3.$
\end{enumerate}
\end{proof}
\begin{note}[Optimal near-completion $\ell$-partition]
{\rm Consider the graph $K_1+C_{21}$ and $V(K_1)=\{v\}.$ From Proposition~\ref{prop2.8}(Part 1) and the fact that $N(v)=V(C_{21})$ prohibits the allocation prescribed by Lemma~\ref{lem2.6} The optimal near-completion $\ell$-partition allows for say $\theta(c_1)=\theta(c_2)=\theta(c_3)=7$ and $\theta(c_4)=1$ say, $c(v)=c_4.$ Clearly for a graph $G$ of order $n\geq 2$ and $\chi(G)\geq 2$ all nested graphs of structure $$\underbrace{K_1+(K_1+(K_1+\cdots +(K_1+ G)))}_{k-times}$$ only an optimal near-completion $\ell$-partition can be found.}
\end{note}
\subsection{Chromatic completion number of helm graphs}
\label{sub2.2}
A helm graph $H_{1,n},$ $n\geq 3.$  is obtained from the wheel graph $W_{1,n}$ by adding a pendant vertex $u_i$ to each rim vertex $v_i$, $1\leq i \leq n.$ Helm graphs derived from wheel graphs, $W_{1,n}$ for even $n,$ will be discussed first. Clearly $n\geq 4.$ Let $\Bbb{N}^{even}=\{n: positive~even~integers, n\geq 4\}.$ Let $\Bbb{N}_1=\{n_i\in \Bbb{N}^{even}:n_i=4+6i, i=0,1,2,\dots\},$ $\Bbb{N}_2=\{n_j \in \Bbb{N}^{even}: n_j=6+6j, j=0,1,2,\dots\}$ and $\Bbb{N}_3=\{n_k \in \Bbb{N}^{even}:n_k=8+6k, k=0,1,2,\dots\}.$ Clearly, $\Bbb{N}^{even} = \Bbb{N}_1 \cup \Bbb{N}_2 \cup \Bbb{N}_3.$
\begin{proposition}
\label{prop2.9}
 Let $H_{1,n_i}$  be a helm graph,  $n_i$ even and  $n_i\geq 4.$ Then
\begin{eqnarray*} 
\zeta(H_{1,n_i}) &=&
\begin{cases}
\frac{(4n_i-1)(n_i-1)}{3}, &\text {if $n_i \in \Bbb{N}_1,$}\\
\frac{n_i(12n_i - 19)}{9}, &\text {if $n_i \in \Bbb{N}_2,$}\\
\frac{12n_i^2 - 27n_i -4}{9}, &\text {if $n_i \in \Bbb{N}_3.$}
\end{cases}
\end{eqnarray*} 
\end{proposition}
\begin{proof}
Part 1: For $n_i \in \Bbb{N}_1$ the colouring $\theta(c_1)=\theta(c_2)=\theta(c_3)=\frac{2n_i+1}{3}$ is always possible. Also, $\varepsilon(H_{1,n_i}) = 3n_i.$ Thus, from Lucky's theorem, Theorem~\ref{thm2.5} read with Corollary~\ref{col2.3} and Lemma~\ref{lem2.6} it follows that, $\zeta(H_{1,n_i}) = 3(\frac{2n_i+1}{3})^2 - 3n_i = \frac{(4n_i-1)(n_i-1)}{3}.$

Part 2:  For $n_i \in \Bbb{N}_2$ the colouring $\theta(c_1)=\theta(c_2)= \lfloor \frac{2n_i+1}{3}\rfloor = \frac{2n_i}{3}$ and $\theta(c_3)=\lceil \frac{2n_i+1}{3}\rceil =\frac{2(n_i+1)}{3}$ is always possible. Also, $\varepsilon(H_{1,n_i}) = 3n_i.$ By similar reasoning as in Part 1, the result of Part 2 follows.

Part 3:  For $n_i \in \Bbb{N}_3$ the colouring $\theta(c_1)=\lfloor \frac{2n_i+1}{3}\rfloor = \frac{2n_i -1}{3}$ and $\theta(c_2)=\theta(c_3)=\lceil \frac{2n_i+1}{3}\rceil =\frac{2(n_i+1)}{3}$ is always possible. Also, $\varepsilon(H_{1,n_i}) = 3n_i.$ By similar reasoning as in Part 1, the result of Part 3 follows.
\end{proof}

The diagrams in Figure~\ref{emb1} serve as illustration of the reasoning used in the proof of Proposition~\ref{prop2.9}

\begin{figure}[htbp]
\begin{center}
\scalebox{0.5}{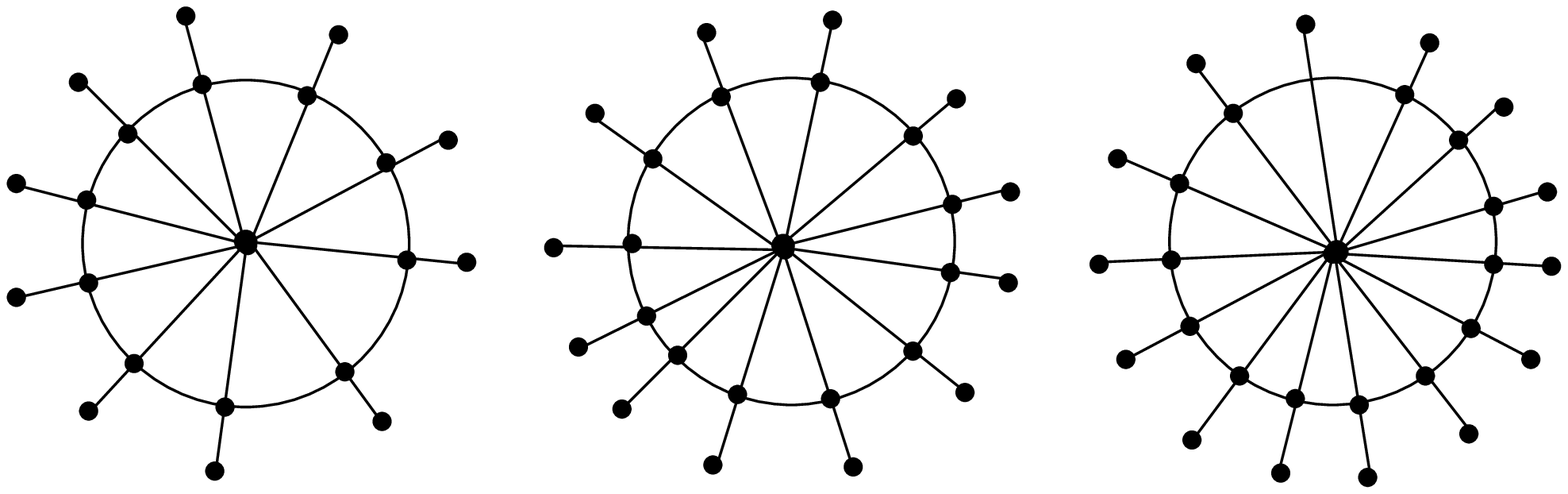}
\caption{}
\label{emb1}
\end{center}
\end{figure}

The next results are for helm graphs $H_{1,n},$ for odd $n.$  Let $\Bbb{N}'_1=\{n_i\in \Bbb{N}^{even}:n_i=3+6i, i=0,1,2,\dots\},$ $\Bbb{N}'_2=\{n_j \in \Bbb{N}^{even}: n_j= 5+6j, j=0,1,2,\dots\}$ and $\Bbb{N}'_3=\{n_k \in \Bbb{N}^{even}:n_k= 7+6k, k=0,1,2,\dots\}.$ Clearly, $\Bbb{N}^{odd} = \Bbb{N}'_1 \cup \Bbb{N}'_2 \cup \Bbb{N}'_3.$
\begin{proposition}
\label{prop2.10}
Let $H_{1,n_i}$  be a helm graph, $n_i$ odd and  $n_i\geq 3.$ Then
\begin{eqnarray*} 
\zeta(H_{1,n_i}) &=&
\begin{cases}
9, &\text {if $n_i =3$},\\
\frac{3n_i(n_i-1)}{2}, &\text {if $n_i \in \Bbb{N}'_1\backslash \{3\}$ or $n_i \in \Bbb{N}'_2$ or $n_i \in \Bbb{N}'_3.$}
\end{cases}
\end{eqnarray*} 
\end{proposition}
\begin{proof}
Part 1: It is easy to verify that $\zeta(H_{1,3})=9.$

Part 2(a): For $n_i \in \Bbb{N}'_1\backslash \{3\}$ the colouring $\theta(c_1)=\theta(c_2)=\theta(c_3)=\lceil \frac{2n_i+1}{4}\rceil = \frac{2(n_i+1)}{4}$ and $\theta(c_4)=\lfloor \frac{2n_i+1}{4}\rfloor =\frac{2(n_i-1)}{4}$ is always possible. Also, $\varepsilon(H_{1,n_i}) = 3n_i.$ Thus, from Lucky's theorem read with Corollary~\ref{col2.3} and Lemma~\ref{lem2.6} it follows that, $\zeta(H_{1,n_i}) = 3(\frac{2(n_i+1)}{4})^2 + 3(\frac{2(n_i+1)\times 2(n_i-1)}{4}) - 3n_i = \frac{3n_i(n_i-1)}{2}.$

Part 2(b): For $n_i \in \Bbb{N}'_2$ the colouring $\theta(c_1)=\theta(c_2)=\theta(c_3)=\lceil \frac{2n_i+1}{4}\rceil = \frac{2(n_i+1)}{4}$ and $\theta(c_4)=\lfloor \frac{2n_i+1}{4}\rfloor =\frac{2(n_i-1)}{4}$ is always possible. Also, $\varepsilon(H_{1,n_i}) = 3n_i.$ This result then follows from Part 2(a) noting, $n_i \in \Bbb{N}'_2.$

Part 2(c): For $n_i \in \Bbb{N}'_3$ the colouring $\theta(c_1)=\theta(c_2)=\theta(c_3)=\lceil \frac{2n_i+1}{4}\rceil = \frac{2(n_i+1)}{4}$ and $\theta(c_4)=\lfloor \frac{2n_i+1}{4}\rfloor =\frac{2(n_i-1)}{4}$ is always possible. Also, $\varepsilon(H_{1,n_i}) = 3n_i.$ This result then follows from Part 2(a) noting, $n_i \in \Bbb{N}'_3.$
\end{proof}

The diagrams in Figure~\ref{emb2} serve as illustration of the reasoning used in the proof of Proposition 2.10.

\begin{figure}[htbp]
\begin{center}
\scalebox{0.5}{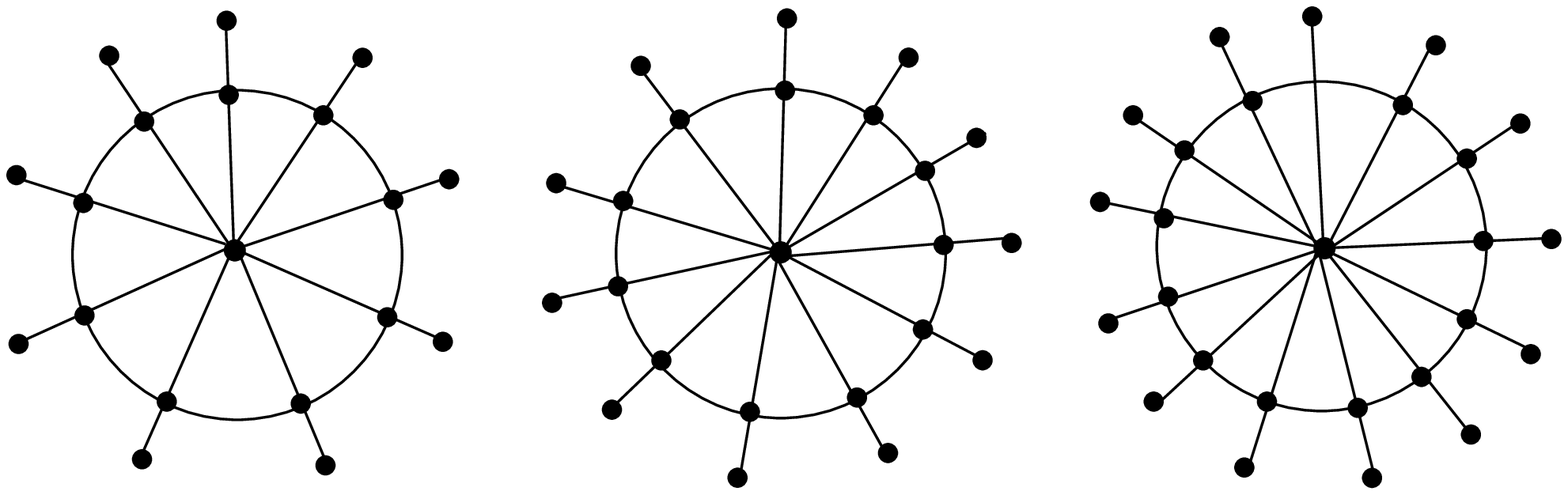}
\caption{}
\label{emb2}
\end{center}
\end{figure}
\section{Conclusion}
\label{s3}
In several of the proofs the technique of graph decomposition permitted by Lemma~\ref{lem2.2} and vertex partitioning permitted by Lemma~\ref{lem2.6} were incorporated. These salient techniques of proof are worthy of further research.

Essentially chromatic completion of a given graph $G$ yields a new graph $G'$ such that both $G,$ $G'$ are of the same order, $\chi(G)=\chi(G'),$  $G\ncong G'$ and $\varepsilon(G')$ is a maximum. For both a chromatic polynomial exists. It is of interest to find a relation between these chromatic polynomials if such relation exists.

Determining the chromatic completion number of a wide range of small graphs is worthy research. Research in respect of all known graph operations remains open. The behavior of chromatic completion for other derivative proper colourings such as Johan colouring (also called $\mathcal{J}$-colouring), co-colouring, Grundy colouring, harmonious colouring, complete colouring, exact colouring, star colouring and others offers a wide scope for further research. Relations between the corresponding derivative chromatic completion numbers, if such exist, are open problems to be investigated. It is suggested that complexity analysis of these new parameters are worthy of further research.

The problem of characterising graphs which permit the colour allocation prescribed by Lemma~\ref{lem2.6} is a challenging open problem. Certainly all graphs $G$ of order $n\geq 4$ with $\chi(G)\geq 2,$ which has a spanning subgraph $H$ which is a star graph, prohibit the prescription of Lemma~\ref{lem2.6}.

%\textbf{\emph{Open access:}} This paper is distributed under the terms of the Creative Commons Attribution License which permits %any use, distribution and reproduction in any medium, provided the original author(s) and the source are credited.

\end{document}